\documentclass[11pt]{amsart}
\usepackage{amssymb,amsmath,amsthm}
\usepackage{a4}
\usepackage{epsf}
\usepackage{epsfig}
\usepackage{url}
\usepackage{hyperref}
\begin{document}

\newcommand{\NP}{$\mathcal{N}\mathcal{P}$}
\newcommand{\newsymb}{{\mathcal P}}
\newcommand{\Pn}{{\mathcal P}^n}
\newcommand{\R}{{\mathbb R}}
\newcommand{\N}{{\mathbb N}}
\newcommand{\Q}{{\mathbb Q}}
\newcommand{\Z}{{\mathbb Z}}
\newcommand{\C}{{\mathbb C}}

\newcommand{\enorm}[1]{\Vert #1\Vert}
\newcommand{\inter}{\mathrm{int}}
\newcommand{\conv}{\mathrm{conv}}
\newcommand{\aff}{\mathrm{aff}}
\newcommand{\lin}{\mathrm{lin}}
\newcommand{\cone}{\mathrm{cone}}
\newcommand{\bd}{\mathrm{bd}}
\newcommand{\stirl}{\mathrm{stirl}}
\newcommand{\ehr}{\mathrm{Ehr}}

\newcommand{\dist}{\mathrm{dist}}
\newcommand{\trans}{\intercal}
\newcommand{\diam}{\mathrm{diam}}
\newcommand{\vol}{\mathrm{vol}}
\newcommand{\F}{\mathrm{F}}
\newcommand{\LE}{\mathrm{G}}
\newcommand{\lE}{\mathrm{g}}
\newcommand{\sa}{\mathrm{a}}

\newcommand{\pp}{\mathfrak{p}}
\newcommand{\pf}{\mathfrak{f}}
\newcommand{\pg}{\mathfrak{g}}
\newcommand{\PP}{\mathfrak{P}}
\newcommand{\pl}{\mathfrak{l}}
\newcommand{\pv}{\mathfrak{v}}
\newcommand{\cl}{\mathrm{cl}}
\newcommand{\bx}{\overline{x}}

\def\ip(#1,#2){#1\cdot#2}

\newtheorem{theorem}{Theorem}[section]
\newtheorem{theorem*}{Theorem}
\newtheorem{corollary}[theorem]{Corollary}
\newtheorem{lemma}[theorem]{Lemma}
\newtheorem{example}[theorem]{Example}
\newtheorem{remark}[theorem]{Remark}
\newtheorem{definition}[theorem]{Definition}  
\newtheorem{conjecture}{Conjecture}[section] 
\newtheorem{proposition}[theorem]{Proposition}  
\newtheorem{claim}[theorem]{Claim}
\newtheorem{problem}[theorem]{Problem}
\numberwithin{equation}{section}

\title[Lower bounds on the coefficients of Ehrhart polynomials]{Lower bounds on the coefficients of Ehrhart polynomials}   
\author{Martin Henk}
\address{Martin Henk, Universit\"at Magdeburg, Institut f\"ur Algebra und Geometrie,
  Universit\"ats\-platz 2, D-39106 Magdeburg, Germany}
\email{henk@math.uni-magdeburg.de}
\author{Makoto Tagami}
\address{Makoto Tagami, Universit\"at Magdeburg, Institut f\"ur Algebra und Geometrie,
  Universit\"ats\-platz 2, D-39106 Magdeburg, Germany}
\email{tagami@kenroku.kanazawa-u.ac.jp}
\thanks{The second author was supported by Research Fellowships of the Japan Society for the Promotion of Science for Young Scientists}

\keywords{Lattice polytopes, Ehrhart polynomial}
\subjclass[2000]{52C07, 52B20, 11H06}

\begin{abstract} 
We present lower bounds for the coefficients of Ehrhart polynomials of convex lattice polytopes  in terms of their volume. Concerning the coefficients of the Ehrhart series of a lattice polytope  we show that Hibi's lower bound  is not true for  lattice polytopes without interior lattice points. The counterexample is based on a formula  of  the Ehrhart series  of the join of  two lattice polytope.  We also present a formula for calculating 
the Ehrhart series of integral dilates of a polytope. 
\end{abstract}

\maketitle

\section{Introduction}
Let $\mathcal{P}^d$ be the set of all convex $d$-dimensional lattice
polytopes in the $d$-dimensional Euclidean space $\R^d$ with respect
to the standard lattice $\Z^d$, i.e., all vertices of $P\in
\mathcal{P}^d$ have  integral coordinates and $\dim(P)=d$. The lattice
point enumerator of a set $S\subset\R^d$, denoted by $\LE(S)$, counts
the number of lattice (integral) points in $S$, 
i.e., $\LE(S)=\#(S\cap \Z^d)$.  
In 1962, Eug\'ene Ehrhart (see e.g.~\cite[Chapter 3]{Beck2007},
\cite{Ehrhart1962}) showed that for $k\in \N$ the lattice point
enumerator $\LE(k\,P)$, $P\in \mathcal{P}^d$, is a polynomial of
degree $d$ in $k$ where the coefficients $\lE_i(P)$, $0\leq i\leq d$,
depend only on $P$: 
\begin{equation}
   \LE(k\,P)=\sum_{i=0}^d \lE_i(P)\,k^i.
\label{eq:ehrhart_polynomial}
\end{equation} 
The polynomial on the right hand side is called the Ehrhart
polynomial, and regarded as a formal polynomial in a complex variable
$z\in\C$ it is  denoted by $\LE_P(z)$. 
Two  of the $d+1$ coefficients
$\lE_i(P)$ are almost obvious, namely, $\lE_0(P)=1$, the Euler characteristic of $P$,  and $\lE_d(P)=\vol(P)$,
where $\vol()$ denotes the volume, i.e., the $d$-dimensional Lebesgue
measure on $\R^d$.  It was shown by Ehrhart (see e.g.~\cite[Theorem 5.6]{Beck2007}, \cite{Ehrhart1967}) that also the second leading coefficient 
admits a simple geometric interpretation  as lattice surface area of 
$P$ 
\begin{equation}
 \lE_{d-1}(P)=\frac{1}{2} \sum_{F\,\text{facet of P}}\frac{\vol_{d-1}(F)}{\det(\aff F\cap\Z^d)}. 
\label{eq:ehrhart_second} 
\end{equation} 
Here $\vol_{d-1}(\cdot)$ denotes the $(d-1)$-dimensional volume  and $\det(\aff F\cap\Z^d)$  denotes the determinant of the $(d-1)$-dimensional sublattice contained in the affine hull of $F$.  
 All other coefficients $\lE_i(P)$, $1\leq i\leq d-2$, have no such
known explicit geometric meaning, except for special classes of polytopes. For this and as a general reference on the theory of lattice polytopes we refer to the recent book of Matthias Beck and Sinai Robins \cite{Beck2007} and the references within. For more information regarding lattices and the role of the lattice point enumerator in convexity see \cite{Gruber1987}. 

In \cite[Theorem 6]{Betke1985a} Ulrich Betke and Peter McMullen proved the following upper bounds on the coefficients  $\lE_i(P)$ in terms of the volume: 
\begin{equation*}
 \lE_i(P)\leq (-1)^{d-i}\stirl(d,i)\vol(P)+(-1)^{d-i-1}\frac{\stirl(d,i+1)}{(d-1)!},\quad i=1,\dots,d-1.
\end{equation*} 
Here $\stirl(d,i)$ denote the Stirling numbers of the first kind which can be defined via the 	identity $\prod_{i=0}^{d-1} (z-i)=\sum_{i=1}^d \stirl(d,i)\,z^i$.

In order to present our lower bounds on $\lE_i(P)$ in terms of the volume we need some notation. For an integer $i$ and a variable $z$ we consider  the polynomial 
\begin{equation*}
  (z+i)(z+i-1)\cdot\ldots\cdot(z+i-(d-1))=d!\,\binom{z+i}{d},    
\end{equation*}  
and we denote its $r$-th coefficient by $C^d_{r,i}$, $0\leq r\leq d$.
For instance, it is $C_{d,i}^d=1$, and for $0\leq i\leq d-1$ we have
$C_{0,i}^d=0$. For $d\geq 3$ we are interested in   
\begin{equation}
        M_{r,d}=\min\{C_{r,i}^d : 1\leq i\leq d-2\}. 
\label{eq:defi}
\end{equation}
Obviously, we have  $M_{0,d}=0$, $M_{d,d}=1$ and it is also easy to see that (cf.~Proposition \ref{prop:special_cases} iii))
\begin{equation}
M_{d-1,d}=C^d_{d-1,1}=-\frac{d(d-3)}{2}.
\label{eq:surface_m}
\end{equation} 
With the help of these numbers $M_{r,d}$ we obtain  the following lower bounds. 
\begin{theorem} Let $P\in \mathcal{P}^d$, $d\geq 3$. Then for $i=1,\dots,d-1$ we have  
\begin{equation*}
  \lE_i(P) \geq \frac{1}{d!}\left\{(-1)^{d-i}\stirl(d+1,i+1)+(d!\,\vol(P)-1)M_{i,d}\right\}.
\end{equation*}
\label{thm:main}
\end{theorem} 
We remark that the coefficients $\lE_i(P)$, $1\leq i\leq d-2$, might be negative  and thus also the lower bounds given above. In general, the bounds of Theorem \ref{thm:main} are not best possible.  For instance,  in the case $i=d-1$ we get together with  \eqref{eq:surface_m}  the bound 
\begin{equation*}
 \lE_{d-1}(P)\geq \frac{1}{(d-1)!}\left\{d-1-\frac{d-3}{2}d!\,\vol(P)\right\}. 
\end{equation*} 
On the other hand, since the lattice surface area of any facet is at least $1/(d-1)!$ we have the trivial inequality (cf.~\eqref{eq:ehrhart_second})
\begin{equation}
\lE_{d-1}(P)\geq \frac{1}{2}\frac{d+1}{(d-1)!}. 
\label{eq:sur_lattice_simplex} 
\end{equation}  
Hence the lower bound on  $\lE_{d-1}(P)$ given in Theorem \ref{thm:main} is only best possible if $\vol(P)=1/d!$.
In the cases $i\in\{1,2,d-2\}$, however, Theorem \ref{thm:main} gives best possible bounds for any volume. 
\begin{corollary} Let $P\in\mathcal{P}^d$, $d\geq 3$. Then 
\begin{equation*}
\begin{split}
{\rm i)}&\quad \lE_1(P)\geq 1+\frac{1}{2}+\cdots +\frac{1}{d-2}+\frac{2}{d-1}-(d-2)!\,\vol(P),\\
{\rm ii)}&\quad \lE_2(P)\geq \frac{(-1)^d}{d!}\times \\ 
         &\quad\quad\quad \left\{\stirl(d+1,3)+\left((-1)^d(d-2)!\,+\stirl(d-1,2)\right)\left(d!\,\vol(P)-1\right)\right\},\\
{\rm iii)}&\quad 
\lE_{d-2}(P) \geq \begin{cases} 
\frac{1}{d!}\frac{(d-1)d(d+1)}{24}\left\{3(d+1)-d!\,\vol(P)\right\} : & \text{if }d \text{ odd},\\[1ex]
 \frac{1}{d!}\frac{(d-1)d}{24}\left\{3d(d+2)-(d-2)\,d!\,\vol(P)\right\} : & \text{if }d \text{ even}.
  \end{cases}
\end{split}   
\end{equation*} 
And the bounds are best possible for any volume. 
\label{cor:special_cases}
\end{corollary}
For some recent inequalities involving more coefficients of Ehrhart
polynomials we refer to \cite{Beck2005}. Next we come to another
family of coefficients of a polynomial associated to lattice
polytopes.  

The generating function of the lattice point enumerator, i.e., the formal power series   
\begin{equation*}
  \ehr_P(z)=\sum_{k\geq 0}\LE_P(k)\,z^k, 
\end{equation*} 
is called the Ehrhart series of $P$. It is well known that it can be expressed as a  rational function of the form 
\begin{equation*}
  \ehr_P(z)=\frac{\sa_0(P)+\sa_1(P)\,z+\cdots+\sa_d(P)\,z^d}{(1-z)^{d+1}}.
\end{equation*} 
The polynomial in the numerator is called the $h^\star$-polynomial. Its   degree is also called the degree of the polytope \cite{Batyrev2007} and it is denoted by $\deg(P)$. 
 Concerning the coefficients $\sa_i(P)$ it is known that they are integral and that 
\begin{equation*}
  \sa_0(P)=1,\quad\sa_1(P)=\LE(P)-(d+1),\quad \sa_d(P)=\LE(\inter(P)),
\end{equation*}
where $\inter(\cdot)$ denotes the interior. Moreover, due to Stanley's
famous   non-negativity theorem (see e.g.~\cite[Theorem
3.12]{Beck2007}, \cite{Stanley1980}) we also know that $\sa_i(P)$ is
non-negative, i.e., for these coefficients we have the lower bounds
$\sa_i(P)\geq 0$. In the case $\LE(\inter(P))>0$, i.e., $\deg(P)=d$,
these bounds  were improved by  Takayuki Hibi \cite{Hibi1994} to 
\begin{equation}
  \sa_i(P) \geq \sa_1(P),\, 1\leq i\leq \deg(P)-1.  
\label{eq:hibi}
\end{equation} 
In this context it was a quite natural  question whether the assumption $\deg(P)=d$ can be weaken (see e.g.~\cite{Nill2007}), i.e., whether these lower bounds \eqref{eq:hibi} are also valid for polytopes of degree less than $d$.   As we show in Example \ref{ex:hibi} the answer is already negative  for polytopes having degree $3$.  
The problem in order to study such a question is that  only very few
geometric constructions of polytopes are known for which we can
explicitly calculate the Ehrhart series. In \cite[Theorem 2.4, Theorem
2.6]{Beck2007} the Ehrhart series of special pyramids and double
pyramids over a basis $Q$ are determined in terms of the Ehrhart
series of $Q$.  In a recent paper Braun   \cite{Braun2006} gave a very
nice product formula for the Ehrhart series of the free sum of two
lattice polytopes, where one of the polytopes has to be
reflexive. Here we consider a related  construction, known as the join of two polytopes \cite{Henk1997a}. As we learned by Matthias Beck the Ehrhart series of such a join  
is already described  as Exercise  3.32 in the book \cite{Beck2007} and it was personally communicated to the authors of the book by Kevin Woods. For completeness' sake we present its short proof in Section 3.
\begin{lemma} For  $P\in\mathcal{P}^p$ and $Q\in\mathcal{P}^q$ let $P\star Q$ be  the join of $P$ and $Q$, i.e., 
\begin{equation*} 
    P\star Q =\conv\left\{(x,0_q,0)^\intercal, (0_p,y,1)^\intercal : x\in P,\,y\in Q\right\}\in\mathcal{P}^{p+q+1}, 
\end{equation*}  
where $0_p$ and $0_q$ denote the $p$- and $q$-dimensional $0$-vector, respectively.  Then 
\begin{equation*}
\ehr_{P\star Q}(z)=\ehr_P(z)\cdot\ehr_Q(z).
\end{equation*} 
\label{lem:one}
\end{lemma}  
In order to apply this lemma we consider two families  of lattice simplices. For an integer $m\in\N$ let  
\begin{equation*} 
\begin{split} 
  T^{(m)}_d & = \conv\{o,e_1,e_1+e_2,e_2+e_3,\dots,e_{d-2}+e_{d-1},e_{d-1}+m\,e_d\}, \\
 S^{(m)}_d & = \conv\{o,e_1,e_2,e_3,\dots,e_{d-1},m\,e_d\},
\end{split} 
\end{equation*} 
where $e_i$ denotes the $i$-th unit vector.  It was shown in \cite{Betke1985a} that 
\begin{equation} 
\ehr_{T^{(m)}_d}(z)= \frac{1+(m-1)\,z^{\lceil\frac{d}{2}\rceil}}{(1-z)^{d+1}} \text{ and } \ehr_{S^{(m)}_d}(z)=\frac{1+(m-1)\,z}{(1-z)^{d+1}}.
\label{eq:examples_t_s}
\end{equation}  
Actually, in \cite{Betke1985a} the formula for $T^{(m)}_d$ was only proved for odd dimensions, but the even case can be treated completely analogously. 
\begin{example}For $q\in\N$ odd  and $l,m\in\N$ we have 
\begin{equation*}
  \ehr_{T^{(l+1)}_q\star  S^{(m+1)}_p}(z)=\frac{1+m\,z+l\,z^\frac{q+1}{2}+m\,l\,z^\frac{q+3}{2}}{(1-z)^{p+q+2}}.
\end{equation*} 
In particular, for $q\geq 3$ and $l<m$ this shows that \eqref{eq:hibi} is, in general,  false for lattice polytopes without interior lattice points.
\label{ex:hibi}
\end{example} 
Another formula for calculating the Ehrhart Series from a given one concerns dilates. Here we will show  
\begin{lemma} Let $P\in\mathcal{P}^d$, $k\in\N$ and let $\zeta$ be a primitive $k$-th root of unity. Then 
\begin{equation*}
   \ehr_{k\,P}(z)=\frac{1}{k}\sum_{i=0}^{k-1}\ehr_P(\zeta^i\,z^\frac{1}{k}).
\end{equation*}
\label{lem:two} 
\end{lemma}   
The lemma can be used, for instance, to calculate the Ehrhart series of the cube $C_d=\{x\in\R^d : |x_i|\leq 1,\,1\leq i\leq d\}$. 
\begin{example} For two integers $j,d$, $0\leq j\leq d$, let 
\begin{equation*}
 A(d,j)=\sum_{k=0}^j (-1)^k\binom{d+1}{k}(j-k)^d
\end{equation*} 
 be the Eulerian numbers (see e.g.~\cite[pp.~28]{Beck2007}).
Furthermore, we set $A(d,j)=0$ if $j\notin\{0,\dots,d\}$. Then, for  $0\leq i\leq d$, we have 
\begin{equation*}
\sa_i(C_d)=\sum_{j=0}^{d+1} \binom{d+1}{j}A(d,2\,i+1-j). 
\end{equation*}  
\label{ex:cube} 
\end{example} 
Of course, the cube $C_d$ may be also regarded as a prism over a $(d-1)$-cube,  
and as a counterpart to the bipyramid construction in \cite{Beck2007} we calculate here also the Ehrhart series of some special prism. 
\begin{example} Let $Q\in\mathcal{P}^{d-1}$, $m\in\N$, and let $P=\{(x,x_d)^\intercal : x\in Q,\, x_d\in [0,m]\}$ be the prism of height $m$ over $Q$. Then  
\begin{equation*}
\sa_i(P)= (m\,i+1)\sa_i(Q)+\left(m(d-i+1)-1\right)\sa_{i-1}(Q),\, 0\leq i\leq d, 
\end{equation*}  
where we set $\sa_d(Q)=\sa_{-1}(Q)=0$.
\label{ex:prism}
\end{example}  
It seems to be quite likely that for the class of $0$-symmetric
lattice polytopes $\mathcal{P}^d_o$ the lower bounds on $\sa_i(P)$ can considerably be improved. In \cite{Bey2007} it was conjectured that for $P\in \mathcal{P}^d_o$ 
\begin{equation*}
  \sa_i(P)+\sa_{d-i}(P)\geq \binom{d}{i}\left(\sa_d(P)+1\right),
\end{equation*}  
where equality holds for instance for the cross-polytopes
$C_d^\star(2\,l-1)=\conv\{\pm l\,e_1,\allowbreak \pm e_i : 2\leq i\leq
d\}$, $l\in\N$, with $2l-1$ interior lattice points. It is also conjectured that these cross-polytopes have minimal volume among all $0$-symmetric lattice polytopes with a given number of interior lattice points. The maximal volume of those polytopes is known by the work of Blichfeldt and van der Corput (cf.~\cite[p.~51]{Gruber1987}) and, for instance, the maximum  is attained by the boxes $Q_d(2\,l-1)=\{|x_1|\leq l,\,|x_i|\leq 1,\, 2\leq i\leq d\}$ with $2\,l-1$ interior points. By the Examples \ref{ex:cube} and \ref{ex:prism} we can easily calculate the Ehrhart series of these boxes.
\begin{example} Let $l\in\N$. Then,  for $0\leq i\leq d$, 
\begin{equation*}
 \sa_i(Q_d(2\,l-1))=(2\,l\,i+1)\,a_i(C_{d-1})+\left(2\,l(d-i+1)-1\right)\,a_{i-1}(C_{d-1}).
\end{equation*} 
\end{example}   
It is quite tempting to conjecture that the box $Q_d(2\,l-1)$ maximizes $\sa_i(P)+\sa_{d-i}(P)$ for $0$-symmetric polytope with $2\,l-1$ interior lattice points. In the $2$-dimensional case this follows easily from a result of Paul Scott \cite{Scott1976a} which implies that $\sa_1(P)\leq 6\,l=\sa_1(Q_2(2\,l-1))$ for any $0$-symmetric convex lattice polygon with $2\,l-1$ interior lattice points.   
In fact, the result of Scott was recently generalized by  Jaron Treutlein \cite{Treutlein2007}  to all degree 2 polytopes.
\begin{theorem}[Treutlein] Let $P\in\mathcal{P}^d$ of degree 2 and let $\sa_i=\sa_i(P)$.  Then 
\begin{equation}
  \sa_1 \leq \begin{cases} 7, & \text{if }\sa_2=1, \\
                              3\,\sa_2+3, & \text{if }\sa_2\geq 2.
                \end{cases} 
\label{eq:treutlein} 
\end{equation} 
\end{theorem} 
In Section 3 we will show that these conditions indeed classify all $h^\star$-polynomials of degree 2.  
\begin{proposition} Let $f(z)=\sa_2\,z^2+\sa_1\,z+1$, $\sa_i\in\N$, satisfying the inequalities in \eqref{eq:treutlein}. Then $f$ is the $h^\star$-polynomial of a lattice polytope. 
\label{prop:class_degree_2}
\end{proposition} 

Concerning lower bounds on the coefficients $\lE_i(P)$ for $0$-symmetric polytopes $P$ we only know, except the trivial case $i=d$, a lower bound on $\lE_{d-1}(P)$ (cf.~\eqref{eq:sur_lattice_simplex}). Namely 
\begin{equation*} 
   \lE_{d-1}(P)\geq \lE_{d-1}(C_d^\star)=\frac{2^{d-1}}{(d-1)!},
\end{equation*} 
where $C_d^\star=\conv\{\pm e_i : 1\leq  i\leq d\}$ denotes the regular cross-polytope. This follows immediately from a result of Richard P.~Stanley \cite[Theorem 3.1]{Stanley1987} on the $h$-vector of ''symmetric'' Cohen-Macaulay simplicial complex. 

Motivated by a problem in \cite{Henk2007} we study in the last section also the related question to bound the surface area $\F(P)$ of a lattice polytope $P$. In contrast to the $\lE_i(P)$'s the surface area is not invariant under unimodular transformations. In order to describe our result we denote by  $T_d$ the standard simplex $T_d=\conv\{0,e_1,\dots,e_d\}$.  

\begin{proposition} Let $P\in\mathcal{P}^d$. Then 
\begin{equation*} 
\begin{split}
   \F(P)\geq \begin{cases}
              &\F(C_d^\star)=\frac{2^d}{d!}\,d^\frac{3}{2}, \text{if } P=-P, \\[1ex]
              &\F(T_d)=\frac{d+\sqrt{d}}{(d-1)!}, \text{ otherwise }. 
             \end{cases} 
\end{split} 
\end{equation*} 
\label{prop:surface_lower}
\end{proposition}

The paper is organized as follows. In the next section we give the
proof  of our main Theorem \ref{thm:main}. Then, in Section 3, we
prove the Lemmas \ref{lem:one} and \ref{lem:two} and show how the
Ehrhart series in the Examples \ref{ex:hibi} and \ref{ex:cube} can be
deduced. Moreover, we will give the proof of Proposition \ref{prop:class_degree_2}. Finally, in the last section we provide a proof of Proposition \ref{prop:surface_lower} which in the symmetric cases is based on a isoperimetric inequality for cross-polytopes (cf.~Lemma \ref{lem:iso_cross}).

\section{Lower bounds on $\lE_i(P)$}
In the following we denote for an integer $r$ and a polynomial $f(x)$ the $r$-th coefficient of $f(x)$, i.e. the coefficient of $x^r$,  by $f(x)|_r$.  
Before proving Theorem \ref{thm:main} we need some basic properties of the numbers $C_{r,i}^d$ and $M_{r,d}$ defined in the introduction (see \eqref{eq:defi}).
We begin with some special cases.
\begin{proposition} Let $d\geq 3$. Then $M_{0,d}=0$, $M_{d,d}=1$ and 
\begin{equation*}
\begin{split}
   {\rm i)}\,& M_{1,d}=C_{1,d-2}^d=-(d-2)!,\\ \quad {\rm ii)}\, &M_{2,d}=C_{2,d-2}^d=(d-2)!\,+(-1)^d\stirl(d-1,2),\\
{\rm iii)}\,&  M_{d-1,d}=C^d_{d-1,1}=-\frac{d(d-3)}{2},\\
{\rm iv)}\,& M_{d-2,d}=\begin{cases}C_{d-2,\frac{d-1}{2}}^d=
  -\frac{1}{4}\binom{d+1}{3}, & \text{if }d \text{ odd},\\C_{d-2,\frac{d}{2}}^d=
  -\frac{1}{4}\binom{d}{3}, & \text{if }d \text{ even}.
                        \end{cases} 
\end{split} 
\end{equation*} 
\label{prop:special_cases}
\end{proposition}
\begin{proof} The cases $M_{0,d}$ and $M_{d,d}$ are trivial. Since $C_{r,l}^d$ is the $(d-r)$-th elementary symmetric function of $\{l,l-1,\dots,l-(d-1)\}$
we have  $C_{1,i}^d=(-1)^{d-i-1}\,i!\,(d-i-1)!$ and
\begin{equation*}
 M_{1,d}=\min\{C_{1,i}^d : 1\leq i\leq d-2\}=C_{1,d-2}^d = -(d-2)!
\end{equation*} 
 In the case $r=2$ we obtain by elementary calculations that 
\begin{equation*}
\begin{split}
 C_{2,i}^d &=i!\,\stirl(d-i,2)+(-1)^d\,(d-i-1)!\,\stirl(i+1,2) \\ 
           &=i!\,(d-i-1)!\,(-1)^{d-i}\left(\sum_{k=1}^{d-i-1}\frac{1}{k}-\sum_{k=1}^i\frac{1}{k}\right),
\end{split} 
\end{equation*} 
from which we conclude $M_{2,d}=C_{2,d-2}^d=(d-2)!\,+(-1)^d\stirl(d-1,2)$. 

For iii) we note that 
\begin{equation*}
  C_{d-1,i}^d=\sum_{j=i-(d-1)}^i j = -\frac{d}{2}(d-1-2\,i),
\end{equation*} 
and so $M_{d-1,d}=C_{d-1,1}^d$. Finally, for the value of $M_{d-2,d}$ we first observe that 
\begin{equation*}
\begin{split}
 C_{d-2,i}^d-C_{d-2,i-1}^d &=
(z+i)\,(z+i-1)\cdot\ldots\cdot(z+i-(d-1))\big|_{d-2}\\
&\hphantom{\hbox{\hspace{1cm}}}-(z+i-1)\cdot\ldots(z+i-(d-1))\,(z+i-d)\big|_{d-2}\\
&=\sum_{j=-d+i+1}^{i-1} j\left(i-(-d+i)\right)=d\sum_{j=-d+i+1}^{i-1} j\\ &=d\frac{(d-1)(-d+2\,i)}{2}.
\end{split} 
\end{equation*} 
Thus the function $C_{d-2,i}^d$ is decreasing in $0\leq i\leq \lfloor d/2\rfloor$ and increasing in $\lfloor d/2\rfloor \leq i\leq d$. So it takes its minimum at $i=\lfloor d/2\rfloor$. First let us assume that $d$ is odd. Then
\begin{equation*}
\begin{split}
 M_{d-2,d}&=C_{d-2,\frac{d-1}{2}}^d = d!\,\binom{z+(d-1)/2}{d}\Bigg|_{d-2} \\ & = z\,(z^2-1)\,(z^2-4)\cdot\ldots\cdot(z^2-((d-1)/2)^2)\bigg|_{d-2}=-\sum_{i=0}^{(d-1)/2}i^2\\ &=-\frac{1}{4}\binom{d+1}{3}.
\end{split} 
\end{equation*}  
The even case can be treated similarly. 
\end{proof} 

In addition  to the previous proposition we also need 
\begin{lemma} \hfill 
\begin{enumerate}
\item[{\rm i)}] $C_{r,i}^d=(-1)^{d-r}C_{r,d-1-i}^d$ for $0\leq i\leq d-1$.
\item[{\rm ii)}] Let $d\geq 3$. Then $M_{r,d}\leq 0$ for $1\leq r\leq d-1$, and $M_{r,d}=0$ only in the case $d=3$ and $r=2$.
\end{enumerate}
\label{lem:help_coeff}
\end{lemma}   
\begin{proof}  The first statement is just a consequence of the fact that $C_{r,l}^d$ is the $(d-r)$-th elementary symmetric function of $\{l,l-1,\dots,l-(d-1)\}$. For ii) we first observe that the case $d=3$ follows directly from Proposition \ref{prop:special_cases}. Hence it remains to show that $M_{r,d}<0$ for $d\geq 4$ and $1\leq r\leq d-1$.   On account of i) it suffices to prove this  when $d-r$ is even and we  will proceed by induction on $d$. 

The case  $d=4$ is covered by Proposition \ref{prop:special_cases}. So let $d\geq 5$. By Proposition \ref{prop:special_cases} i)  we also may assume $r\geq 2$.  It is easy to see that 
\begin{equation}
   C_{r,i}^d = \left(i-d+1\right)\,C_{r,i}^{d-1}+C_{r-1,i}^{d-1},
\label{eq:recur}
\end{equation} 
and by induction we may assume that there exists a $j\in\{1,\dots,d-3\}$ with 
$C_{r-1,j}^{d-1}< 0$. Observe that $d-1-(r-1)$ is even. If
$C_{r,j}^{d-1}\geq 0$ we obtain by \eqref{eq:recur} that $C_{r,j}^d< 0$ and we are done. So  let 
$C_{r,j}^{d-1}<0$. By part i) we know that 
\begin{equation*}
 C_{r,j}^{d-1}=(-1)^{d-1-r}C_{r,d-2-j}^{d-1} \text{ and } C_{r-1,j}^{d-1}=(-1)^{d-r}\,C_{r-1,d-2-j}^{d-1}.
\end{equation*} 
Since $d-r$ is even we conclude $C_{r,d-2-j}^{d-1}>0$ and $C_{r-1,d-2-j}^{d-1}< 0$. Hence, on account of \eqref{eq:recur} we get $C_{r,d-2-j}^d< 0$ and so $M_{r,d}< 0$.
\end{proof} 

Now we are able to give the proof of our main Theorem. 
\begin{proof}[Proof of Theorem \ref{thm:main}] We follow the approach
  of Betke and McMullen used in \cite[Theorem
  6]{Betke1985a}. By expanding the Ehrhart series at $z=0$ one gets (see e.g.~\cite[Lemma 3.14]{Beck2007})
\begin{equation}
 \LE_P(z)=\sum_{i=0}^d\sa_i(P)\binom{z+d-i}{d}.
\label{eq:expand} 
\end{equation} 
In particular, we have 
\begin{equation}
    \frac{1}{d!}\sum_{i=0}^d\sa_i(P)=\lE_d(P)=\vol(P). 
\label{eq:volume} 
\end{equation} 
For short, we will write $\sa_i$ instead of $\sa_i(P)$ and $\lE_i$ instead of $\lE_i(P)$. With this notation we have 
\begin{equation}
\begin{split}
 d!\,\lE_r&=d!\,\LE_P(z)|_r =d!\,\sum_{i=0}^d\sa_i\binom{z+d-i}{d}\Bigg|_r \\&=C_{r,d}^d+(\sa_1\,C_{r,d-1}^d+\sa_d\,C_{r,0}^d)+\sum_{i=2}^{d-1} \sa_i\,C_{r,d-i}^d.
\end{split} 
\label{eq:ineq_1}
\end{equation}
Since $C_{r,d-1}^d\geq 0$ we get with Lemma \ref{lem:help_coeff} i) that $C_{r,d-1}^d=|C_{r,0}^d|$. Together with $\sa_1=\LE(P)-(d+1)\geq \LE(\inter(P))=\sa_d$ and $C_{r,d}^d=(-1)^{d-r}\stirl(d+1,r+1)$ we find 
\begin{equation}
\begin{split} 
d!\,\lE_r &\geq  (-1)^{d-r}\stirl(d+1,r+1) +\sum_{i=2}^{d-1}\sa_i\,C_{r,d-i}^d \\
&= (-1)^{d-r}\stirl(d+1,r+1) +\sum_{i=2}^{d-1}\sa_i\,\left(C_{r,d-i}^d-M_{r,d}\right) +\sum_{i=1}^{d}\sa_i\,M_{r,d}\\ 
&\hphantom{\hbox{\hspace{5cm}}}-(\sa_1+\sa_d)M_{r,d}\\
&\geq (-1)^{d-r}\stirl(d+1,r+1) +(d!\,\vol(P)-1)M_{r,d},
\end{split} 
\label{eq:ineq_2}
\end{equation} 
where the last inequality follows from the definition of $M_{r,d}$ and the non-positivity of $M_{r,d}$ (cf.~Proposition \ref{prop:special_cases} and  Lemma \ref{lem:help_coeff} ii)).
\end{proof}

We remark that for $d\geq 3$, $r\in\{1,\dots,d-1\}$ and $(r,d)\ne(2,3)$  we can slightly improve the inequalities in Theorem \ref{thm:main}, because in these cases we have $M_{r,d}<0$ (cf.~Lemma \ref{lem:help_coeff} ii)),  and since $C_{r,d-1}^d$ is the $(d-r)$-th elementary symmetric function of $\{0,\dots,d-1\}$ we also know  $C_{r,d-1}^d>0$ for $1\leq r\leq d-1$. Hence   we get  (cf.~\eqref{eq:ineq_1} and \eqref{eq:ineq_2})  
\begin{equation*}
\begin{split}
  d!\,\lE_r & = C_{r,d}^d+\sum_{i=1}^{d} \sa_i\,C_{r,d-i}^d \\ 
  & = C_{r,d}^d + \sa_1\left(C_{r,d-1}^d-M_{r,d}\right)+\sum_{i=2}^{d}\left(C_{r,d-i}^d-M_{r,d}\right) +\sum_{i=1}^{d}\sa_i\,M_{r,d} \\ 
&\geq 
 (-1)^{d-r}\stirl(d+1,r+1) +2\,\sa_1(P)+(d!\,\vol(P)-1)M_{r,d} \\
 &=(-1)^{d-r}\stirl(d+1,r+1) -2(d+1) + 2\LE(P)+(d!\,\vol(P)-1)M_{r,d}.
\end{split} 
\end{equation*}

Corollary \ref{cor:special_cases} is an immediate consequence of Theorem \ref{thm:main} and Proposition \ref{prop:special_cases}. 

\begin{proof}[Proof of Corollary \ref{cor:special_cases}] The inequalities just follow by inserting the value of $M_{r,d}$ given in Proposition \ref{prop:special_cases} 
in the general inequality of Theorem \ref{thm:main}. Here we also have used the identities   
\begin{equation*}
 \stirl(d+1,2)=(-1)^{d+1}\,d!\,\sum_{i=1}^d \frac{1}{i}\text{ and }
\stirl(d+1,d-1)=\frac{3\,d+2}{4}\binom{d+1}{3}.
\end{equation*} 
It remains to show that the inequalities are best possible for any volume.
For $r=d-2$ we consider the simplex $T^{(m)}_d$ (cf.~\eqref{eq:examples_t_s}) with  $\sa_0(T^{(m)}_d)=1$, $\sa_{\lceil d/2\rceil}(T^{(m)}_d)=(m-1)$ and $\sa_i(T^{(m)}_d)=0$ for $i\notin\{0,\lceil d/2\rceil \}$. Then $\vol(T^{(m)}_d)=m/d!$ and on account of Proposition \ref{prop:special_cases} we have equality in \eqref{eq:ineq_1} and \eqref{eq:ineq_2}. 

For $r=1,2$ and $d\geq 4$ we consider the $(d-4)$-fold pyramid $\tilde{T}^{(m)}_d$ over $T^{(m)}_4$ given by $\tilde{T}^{(m)}_d=\conv\{T^{(m)}_4,e_5,\dots,e_d\}$. Then $\vol(\tilde{T}^{(m)}_d)=m/d!$ and in view of \eqref{eq:examples_t_s} and \cite[Theorem 2.4]{Beck2007} we obtain 
\begin{equation*}
\sa_0(\tilde{T}^{(m)}_d)=1,\, \sa_2(\tilde{T}^{(m)}_d)=m-1 \text{ and } 
\sa_i(\tilde{T}^{(m)}_d)=0,\,i\notin\{0,2\}.
\end{equation*} 
Again, by  Proposition \ref{prop:special_cases} we have equality in \eqref{eq:ineq_1} and \eqref{eq:ineq_2}.
\end{proof}

\section{Ehrhart series of some special polytopes} 
We start with the short proof of Lemma \ref{lem:one}.
\begin{proof}[Proof of Lemma \ref{lem:one}] Since 
\begin{equation*}
 \ehr_P(z)\,\ehr_Q(z)=\sum_{k\geq 0} \left(\sum_{m+l=k} \LE_P(m)\LE_Q(l)\right)z^k, 
\end{equation*} 
it suffices to prove that the Ehrhart polynomial $\LE_{P\star Q}(k)$ of the lattice polytope $P\star Q\in\mathcal{P}^{p+q+1}$ is given by 
\begin{equation*}
\LE_{P\star Q}(k)= \sum_{m+l=k} \LE_P(m)\LE_Q(l).
\end{equation*} 
This, however, follows immediately from the definition since 
\begin{equation*}
k\,(P\star Q) =\left\{\lambda\,(x,o_q,0)^\intercal +(k-\lambda)\,(o_p,y,1)^\intercal : x\in P,\,y\in Q, 0\leq \lambda\leq k\right\}.
\end{equation*} 
\end{proof} 
Example \ref{ex:hibi} in the introduction shows an application of this construction. For Example \ref{ex:cube} we need Lemma \ref{lem:two}.

\begin{proof}[Proof of Lemma \ref{lem:two}] 
With $w=z^\frac{1}{k}$ we may write 
\begin{equation*}
\frac{1}{k}\sum_{i=0}^{k-1}\ehr_P(\zeta^i\,w)=\frac{1}{k}\sum_{i=0}^{k-1} \sum_{m\geq 0} \LE_P(m)(\zeta^i\,w)^m = \frac{1}{k} \sum_{m\geq 0} \LE_P(m)w^m\sum_{i=0}^{k-1} \zeta^{i\,m}.
\end{equation*} 
Since $\zeta$ is a $k$-th root of unity the sum $\sum_{i=0}^{k-1} \zeta^{i\,m}$ is equal to $k$ if $m$ is a multiple of $k$ and otherwise it is $0$. Thus we obtain 
\begin{equation*}
 \frac{1}{k}\sum_{i=0}^{k-1}\ehr_P(\zeta^i\,w)=\sum_{m\geq 0} \LE_P(m\,k)w^{m\,k}=\sum_{m\geq 0} \LE_{k\,P}(m)z^{m}=\ehr_{k\,P}(z).
\end{equation*} 
\end{proof} 
As an application of Lemma \ref{lem:two} we calculate the Ehrhart series of the cube $C_d$ (cf.~Example \ref{ex:cube}). Instead of $C_d$ we  consider the translated cube $2\,\tilde{C}_d$, where $\tilde{C}_d=\{x\in\R^d: 0\leq x_i\leq 1,\,1\leq i\leq d\}$. In \cite[Theorem 2.1]{Beck2007} it was shown that $\sa_i(\tilde{C}_d)=A(d,i+1)$ where $A(d,i)$ denotes the Eulerian numbers. Setting $w=\sqrt{z}$ Lemma \ref{lem:two} leads to 
\begin{equation*} 
\begin{split}
\ehr_{C_d}(z)& =\frac{1}{2}\left(\ehr_{\tilde{C}_d}(w)+\ehr_{\tilde{C}_d}(-w)\right)\\
&=\frac{1}{2}\left(\frac{\sum_{i=1}^{d} A(d,i)\,w^{i-1}}{(1-w)^{d+1}}+\frac{\sum_{i=1}^{d} A(d,i)\,(-w)^{i-1}}{(1+w)^{d+1}}\right)\\
&=\frac{1}{2}\frac{1}{(1-z)^{d+1}}
\Bigg(\sum_{i=1}^{d} A(d,i)\,w^{i-1}\,(1+w)^{d+1} \\ 
  &\hphantom{\hbox{\hspace{3cm}}}+\sum_{i=1}^{d} A(d,i)\,(-w)^{i-1}\,(1-w)^{d+1}\Bigg) \\
&=\frac{1}{(1-z)^{d+1}}\left(\sum_{i=1}^d A(d,i)\sum_{j=0, \text{ $i+j-1$ even}}^{d+1}\binom{d+1}{j}\,w^{i+j-1}\right)
\end{split} 
\end{equation*} 
Substituting $2\,l=i+j-1$ gives 
\begin{equation*}
\begin{split}
\ehr_{C_d}(z)& = \frac{1}{(1-z)^{d+1}}\left(\sum_{l=0}^d \sum_{i=2\,l-d}^{2\,l+1}\binom{d+1}{2\,l+1-i}\,A(d,i)\,w^{2\,l}\right) \\
&=\frac{1}{(1-z)^{d+1}}\left(\sum_{l=0}^d z^l\,\sum_{j=0}^{d+1}\binom{d+1}{j}\,A(d,2\,l+1-j)\right),
\end{split} 
\end{equation*} 
which explains the formula in Example \ref{ex:cube}. 

In order to calculate in general the Ehrhart series of the prism $P=\{(x,x_d)^\intercal : x\in Q, x_d\in[0,m]\}$  where $Q\in\mathcal{P}^{d-1}$, $m\in\N$ (cf.~Example \ref{ex:prism}), we use the differential operator  $T$ defined by $z \frac{\rm d}{\rm d z}$. Considered as an operator on the ring of formal power series we have (cf.~e.g.~\cite[p.~28]{Beck2007})  
\begin{equation}
 \sum_{k\geq 0} f(k)\,z^k = f(T)\frac{1}{1-z} 
\label{eq:operator} 
\end{equation}
for any polynomial $f$.  Since $\LE_P(k)=(m\,k+1)\,\LE_Q(k)$ we deduce from \eqref{eq:operator}
\begin{equation*}
\ehr_P(z)=(m\,T+1)\ehr_Q(z) = m z \frac{\rm d}{\rm d z}\ehr_Q(z)+ \ehr_Q(z).
\end{equation*} 
Thus 
\begin{equation*}
\begin{split}
\ehr_P(z)&=m\,z\,\frac{\sum_{i=0}^{d-1}i\,\sa_i(Q)z^{i-1}(1-z)+\sum_{i=0}^{d-1} d\,\sa_i(Q)\,z^i}{(1-z)^{d+1}}+\frac{\sum_{i=0}^{d-1}\sa_i(Q)\,z^i}{(1-z)^d} \\ 
&=\frac{\sum_{i=0}^{d-1}(m\,i+1)\sa_i(Q)z^i(1-z)+\sum_{i=0}^{d-1}m\,d\,\sa_i(Q)z^{i+1}}{(1-z)^{d+1}}\\
&=\frac{1}{(1-z)^{d+1}} \sum_{i=1}^d \left((m\,i+1)\sa_i(Q)+\left(m(d-i+1)-1\right)\sa_{i-1}(Q)\right)z^i,
\end{split}
\end{equation*} 
which is the formula in Example \ref{ex:prism}. 

Finally, we come to the classification of $h^\star$-polynomials of degree 2.

\begin{proof}[Proof of Proposition \ref{prop:class_degree_2}] We recall that $\sa_1(P)=\LE(P)-(d+1)$ and $\sa_d(P)=\LE(\inter(P))$ for $P\in\mathcal{P}^d$. 
In the case $\sa_2=1$, $\sa_1=7$ the triangle $\conv\{0,3\,e_1,3\,e_2\}$ has the desired $h^\star$-polynomial. Next we distinguish two cases:
\begin{itemize}
\item[i)] $\sa_2 < \sa_1 \leq 3\,\sa_2+3$. For integers $k,l,m$ with 
$0\leq l,k\leq m+1$ let $P\in\mathcal{P}^2$ given by $P=\conv\{0,l\,e_1,e_2+(m+1)\,e_1,2\,e_2,2\,e_2+k\,e_1\}$. Then it is easy to see that $\sa_2(P)=m$ and $P$ has $k+l+4$ lattice points on the boundary. Thus $\sa_1(P)=k+l+m+1$.
\item[ii)] $\sa_1\leq \sa_2$. For integers $l,m$ with $0\leq l\leq m$ let $P\in\mathcal{P}^3$ given by $P=\conv\{0,e_1,e_2,-l\,e_3,e_1+e_2+(m+1)\,e_3\}$. The only lattice points contained in $P$ are the vertices and the lattice points on the edge $\conv\{0,-l\,e_3\}$. Thus $\sa_3(P)=0$ and $\sa_1(P)=l$. On the other hand, since $(l+m+1)/6=\vol(P)=(\sum_{i=0}^3 \sa_i(P))/6$ (cf.~\eqref{eq:volume}) it is $\sa_2(P)=m$.
\end{itemize}
\end{proof}

\section{$0$-symmetric lattice polytopes}
In order to study the surface area of $0$-symmetric polytopes we first prove an isoperimetric inequality for the class of cross-polytopes. 
\begin{lemma} Let $v_1,\dots,v_d\in\R^d$ be linearly independent 
and let $C=\conv\{\pm v_i : 1\leq i\leq d\}$. Then 
\begin{equation*}
    \frac{\F(C)^d}{\vol(C)^{d-1}} \geq \frac{2^d}{d!}\,d^{\frac{3}{2}d},
\end{equation*}  
and equality holds if and only if $C$ is a regular cross-polytope, i.e., $v_1,\dots,v_d$ form an orthogonal basis of equal length. 
\label{lem:iso_cross}
\end{lemma} 
\begin{proof} Without loss of generality let $\vol(C)=2^d/d!$. Then we have to show   
\begin{equation}
\F(C)\geq \frac{2^d}{d!}d^\frac{3}{2}.
\label{eq:to_show_cross}
\end{equation}
By standard arguments from convexity (see e.g.~\cite[Theorem 6.3]{Gruber2007}) the set of all $0$-symmetric cross-polytopes with volume $2^d/d!$ contains a cross-polytope $C^\star=\conv\{\pm w_1,\dots,\pm w_d\}$, say, of minimal surface area. 
Suppose that some of the vectors are not pairwise orthogonal, for instance, $w_1$ and $w_2$.  Then we apply to $C^\star$ a Steiner-Symme\-trization (cf.~e.g.~\cite[pp.~169]{Gruber2007}) with respect to the hyperplane $H=\{x\in\R^d : w_i\,x=0\}$. It is easy to check that the Steiner-symmetral of $C^\star$ is again a cross-polytope $\tilde{C}^*$, say,  with $\vol(\tilde{C}^\star)=\vol(C^\star)$  (cf.~\cite[Proposition 9.1]{Gruber2007}). Since $C^\star$ was not symmetric with respect to the hyperplane $H$ we also know that $\F(\tilde{C}^*)<\F(C^\star)$ which contradicts the minimality of $C^\star$  (cf.~\cite[p.~171]{Gruber2007}). 

So we can assume that the vectors $w_i$ are pairwise orthogonal. Next suppose that $\enorm{w_1}>\enorm{w_2}$, where $\enorm{\cdot}$ denotes the Euclidean norm. Then we apply Steiner-Symmetrization with respect to the hyperplane $H$ which is orthogonal to $w_1-w_2$ and bisecting the edge $\conv\{w_1,w_2\}$. As before we get a contradiction to the minimality of $C^\star$. 

Thus we know that $w_i$ are pairwise orthogonal and of same length. By our assumption on the volume we get $\enorm{w_i}=1$, $1\leq i\leq d$, and it is easy to calculate that $\F(C^\star)=(2^d/d!)d^{3/2}$. So we have 
\begin{equation*}
 \F(C)\geq \F(C^\star)=\frac{2^d}{d!}d^\frac{3}{2}, 
\end{equation*} 
and by the foregoing argumentation via Steiner-Symmetrizations we also see that equality holds if and only $C$ is a regular cross-polytope generated by vectors of unit-length.  
\end{proof} 

The determination of the minimal surface area of $0$-symmetric lattice polytopes is an immediate consequence of the lemma above, whereas the non-sym\-metric case does not follow from the corresponding  isoperimetric inequality for simplices.    
\begin{proof}[Proof of Proposition \ref{prop:surface_lower}] Let $P\in\mathcal{P}^d$ with $P=-P$. Then $P$ contains a $0$-symmetric lattice cross-polytope $C=\conv\{\pm v_i :1\leq i\leq d\}$, say, and by the monotonicity of the surface area and Lemma \ref{lem:iso_cross} we get 
\begin{equation}
  \F(P)\geq \F(C)\geq \left(\frac{2^d}{d!}\right)^\frac{1}{d}\,d^\frac{3}{2}\,\vol(C)^\frac{d-1}{d}. 
\label{eq:ineq_one} 
\end{equation}  
Since $v_i\in\Z^d$, $1\leq i\leq d$,  we have $\vol(C)=(2^d/d!)|\det(v_1,\dots,v_d)|\geq 2^d/d!$, which shows by \eqref{eq:ineq_one} the $0$-symmetric case.

In the non-symmetric case we know that $P$ contains a lattice simplex $T=\{x\in\R^d : a_i\,x \leq b_i,\,1\leq i\leq d+1\}$, say. Here we may assume that $a_i\in\Z^n$ are primitive, i.e., $\conv\{0,a_i\}\cap\Z^n=\{0,a_i\}$, and that $b_i\in\Z$. Furthermore, we denote the facet 
$P\cap\{x\in\R^d : a_i\,x=b_i\}$ by $F_i$, $1\leq i\leq d+1$. With these notations we have $\det(\aff F_i\cap\Z^n)=\enorm{a_i}$ (cf.~\cite[Proposition 1.2.9]{Martinet2003}). Hence there exist integers $k_i\geq 1$ with 
\begin{equation}
   \vol_{d-1}(F_i)=k_i\,\frac{\enorm{a_i}}{(d-1)!},
\label{eq:sur_norm}
\end{equation} 
and so we may write 
\begin{equation*}
\F(P)\geq \F(T)=\sum_{i=1}^{d+1}\vol_{d-1}(F_i) \geq \frac{1}{(d-1)!}\sum_{i=1}^{d+1} \enorm{a_i}.
\end{equation*} 
We also have $\sum_{i=1}^{d+1}\vol_{d-1}(F_i)a_i/\enorm{a_i}=0$ (cf.~e.g.~\cite[Theorem 18.2]{Gruber2007}) and in view of \eqref{eq:sur_norm} we obtain $\sum_{i=1}^{d+1}k_i\,a_i=0$. Thus, since the $d+1$ lattice vectors $a_i$ are affinely independent we can find for each index $j\in\{1,\dots,d\}$ at least two vectors $a_{i_1}$ and $a_{i_2}$ having a non-trivial $j$-th coordinate. Hence 
\begin{equation}
 \sum_{i=1}^{d+1} \enorm{a_i}^2 \geq 2\,d.
\label{eq:ineq_two}
\end{equation} 
Together with the restrictions $\enorm{a_i}\geq 1$, $1\leq i\leq d+1$, it is easy to argue  that $\sum_{i=1}^{d+1} \enorm{a_i}$ is minimized if and only if $d$ norms $\enorm{a_i}$ are equal to 1 and one is equal to $\sqrt{d}$. For instance, the intersection of the cone $\{x\in\R^{d+1} : x_i\geq 1, 1\leq i\leq d+1\}$ with the hyperplane $H_\alpha=\{x\in\R^{d+1} : \sum_{i=1}^{d+1} x_i=\alpha\}$, $\alpha\geq d+1$, is the  $d$-simplex 
$T(\alpha)$ with vertices given by the permutations of the vector $(1,\dots,1, \alpha-d)^\intercal$ of  length $\sqrt{d+(\alpha-d)^2}$. Therefore, a vertex of that simplex is contained in $\{x\in\R^{d+1} : \sum_{i=1}^{d+1}x_i^2\geq 2d\}$ if  $\alpha\geq d+\sqrt{d}$. In other words, we always have 
\begin{equation*}
\sum_{i=1}^{d+1}\enorm{a_i} \geq d+\sqrt{d},
\end{equation*} 
which gives the desired inequality in the non-symmetric case (cf.~\eqref{eq:sur_norm}). 
\end{proof} 
We remark that the proof also shows that equality in Proposition \ref{prop:surface_lower} holds if and only if $P$ is the $o$-symmetric cross-polytope $C_d^\star$ or the simplex $T_d$ (up to lattice translations).

\medskip
\noindent
{\it Acknowledgement.} The authors would like to thank Matthias Beck, Benjamin Braun, Christian Haase and the anonymous referee for valuable comments and suggestions. 
\def\cprime{$'$}
\providecommand{\bysame}{\leavevmode\hbox to3em{\hrulefill}\thinspace}
\providecommand{\MR}{\relax\ifhmode\unskip\space\fi MR }
\providecommand{\MRhref}[2]{%
  \href{http://www.ams.org/mathscinet-getitem?mr=#1}{#2}
}
\providecommand{\href}[2]{#2}

\end{document}